\documentclass{amsart}
\usepackage{amssymb}
\usepackage{graphicx}
\usepackage[all]{xy}

\newtheorem{theorem}{Theorem}[section]
\newtheorem{lemma}[theorem]{Lemma}
\newtheorem{proposition}[theorem]{Proposition}
\newtheorem{corollary}[theorem]{Corollary}

\numberwithin{equation}{section}

\begin{document}

\title{Schreier spectrum of the {H}anoi {T}owers group on three pegs}

\author{Rostislav Grigorchuk}

\email{grigorch@math.tamu.edu}

\author{Zoran \v Suni\'c}

\email{sunic@math.tamu.edu}

\address{Department of Mathematics, Texas A\&M University, MS-3368, College Station, TX 77843-3368,
USA}

\thanks{Both authors would like to announce the support by NSF grants
DMS-0600975 and DMS-0648786 and by Isaac Newton Institute for
Mathematical Sciences in Cambridge, UK. The first author was also
supported by NSF grant DMS-0456185}

\date{}


\begin{abstract}
Finite dimensional representations of the Hanoi Towers group are
used to calculate the spectra of the finite graphs associated to the
Hanoi Towers Game on three pegs (the group serves as a renorm group
for the game). These graphs are Schreier graphs of the action of the
Hanoi Towers group on the levels of the rooted ternary tree. The
spectrum of the limiting graph (Schreir graph of the action on the
boundary of the tree) is also provided.
\end{abstract}

\maketitle


\section{Introduction}

The spectral theory of graphs and groups has links to many areas of
mathematics. It allows for a combinatorial approach to the spectral
theory of Laplace-Beltrami operators on manifolds, but at the same
time it often gives completely new approach to some problems in
algebra, operator algebras, random walks, and combinatorics. It is
also closely related to many topics in fractal geometry.

The goal of this note is to show how algebra (more precisely, a
group structure) may be used to solve the spectral problem for the
sequence of graphs that are modifications of Sierpi\'{n}ski graphs
and which naturally arise in the study of the popular combinatorial
problem often called Hanoi Towers Game or Hanoi Towers Problem (for
a survey of topics and results related to Hanoi Towers Game
see~\cite{hinz:ens}.

It is indicated in~\cite{grigorchuk-s:hanoi-cr}
and~\cite{grigorchuk-s:standrews} that the Hanoi Towers Game on $k$
pegs, $k \geq 3$, can be modeled by a self-similar group, denoted
$H^{(k)}$, generated by a finite automaton on $k(k-1)/2+1$ states
over the alphabet $X=\{0,1\dots,k-1\}$ of cardinality $k$. The group
$H^{(k)}$ acts by automorphisms on the rooted $k$-regular tree and
one may consider the sequence of (finite) Schreier graphs
$\{\Gamma_n^{(k)}\}$ (graphs of the action, or orbit graphs) of the
action of $H^{(k)}$ on level $n$ of the $k$-ary tree as well as the
(infinite) graph $\Gamma^{(k)}$ corresponding to the action of
$H^{(k)}$ on the orbit of a point on boundary of the tree (the
choice of the point does not play a role).

The combinatorial problem known as Hanoi Towers Game (on $k$ pegs,
$k \geq 3$, and $n$ disks, $n \geq 0$) can be reformulated as the
problem of finding the distance (or at least a good asymptotic
estimate and/or an algorithm for finding the shortest path) between
particular vertices in the graph $\Gamma_n^{(k)}$ (namely between
the vertices $0^n$ and $1^n$ in the natural encoding of the vertices
at level $n$ in the $k$-ary tree by words of length $n$ in the
alphabet $X=\{0,1,..., k-1\}$). This problem is closely related
(although not equivalent) to the problem of finding the diameters of
the graphs $\Gamma_n^{(k)}$, $n \geq 0$.

Questions on distances and diameters are related to the spectral
analysis of the involved graphs. It is known that the spectrum of a
graph does not determines the graph completely (the famous question
of Mark Kac ``Can one hear ...'' has a negative answer in the
context of graphs as well). Nevertheless, the spectrum gives a lot
of information about the graph. An upper bound on the diameter can
be obtained in terms of the second largest eigenvalue
(see~\cite{chung:b-spectral} for results and references in this
direction).

Therefore, one approach to Hanoi Towers Problem is to first try, for
fixed $k$, to solve the spectral problem for the sequence of graphs
$\{\Gamma_n^{(k)}\}$ and then use the obtained information to
compute or provide asymptotics for the diameters and distances
between $0^n$ and $1^n$ (note that Szegedy has already established
such asymptotics by using more direct methods~\cite{szegedy:hanoi}).
For $k=3$, the Hanoi Towers Problem has been solved in many ways
(both recursive and iterative). Namely, the diameter and the
distance between $0^n$ and $1^n$ in $\Gamma_n^{(3)}$ coincide and
are equal to $2^n-1$. On the other hand, the problem is open for $k
\geq 4$. It has long been conjectured (and ``proven'' many times)
that Frame-Stewart algorithm~\cite{frame:solution,stewart:solution}
does provide an optimal solution. This algorithm takes roughly
$2^{n^{\frac{1}{k-2}}}$ steps and all that can be safely concluded
after the work of Szegedy~\cite{szegedy:hanoi} is that this
algorithm is asymptotically optimal for $k \geq 4$.

True to form, we only have a full solution to the spectral problem
in case $k=3$ (announced in~\cite{grigorchuk-s:hanoi-cr}). Some
hints on the method used to calculate the spectrum are given
in~\cite{grigorchuk-s:standrews} as well as
in~\cite{grigorchuk-n:schur}. Here we give the complete proof and
use the opportunity to point out some intricacies of the approach.

In a sense, our approach is algebraic and is based on the fact that
Hanoi Towers group $H^{(k)}$ serves as the renorm group for the
model. We would like to stress that $H^{(k)}$ is not the group of
symmetries of the obtained Schreier graph(s), but rather the renorm
group of the model (the renorm aspect of self-similar groups is
emphasized in the review~\cite{grigorchuk:review-ss}).

Indeed many regular graphs can be realized as Schreier graphs of
some group. The symmetries that are described by the group are not
the symmetries of the system but nevertheless may be used to solve
the spectral problem. This viewpoint was initiated
in~\cite{bartholdi-g:spectrum} and further developed
in~\cite{grigorchuk-z:l2} and~\cite{kambites-s-s:spectra}. An
interesting phenomenon is the appearance of the Schur complement
transformation~\cite{grigorchuk-s-s:z2i,grigorchuk-n:schur}. One of
the main features of the method is the use of operator recursions
coming from the self-similarity structure, followed by algebraic
manipulations (sometimes employing $C^*$-algebra techniques).

We now describe our main result.

Let $\Gamma_n$ be the Schreier graph of the Hanoi Towers group
$H^{(3)}$ on three pegs ($\Gamma_3$ is depicted
in~Figure~\ref{f:g3}; see the next section for a precise
definition). Further, Let $\Gamma$ be the orbital Schreier graph of
$H^{(3)}$ corresponding to the orbit of the infinite word $000\dots$
(this graph is a limit of the sequence of graphs $\{\Gamma_n\}$; see
the next section for a precise definition).
\begin{figure}[!ht]
\begin{center}
\includegraphics[width=280pt]{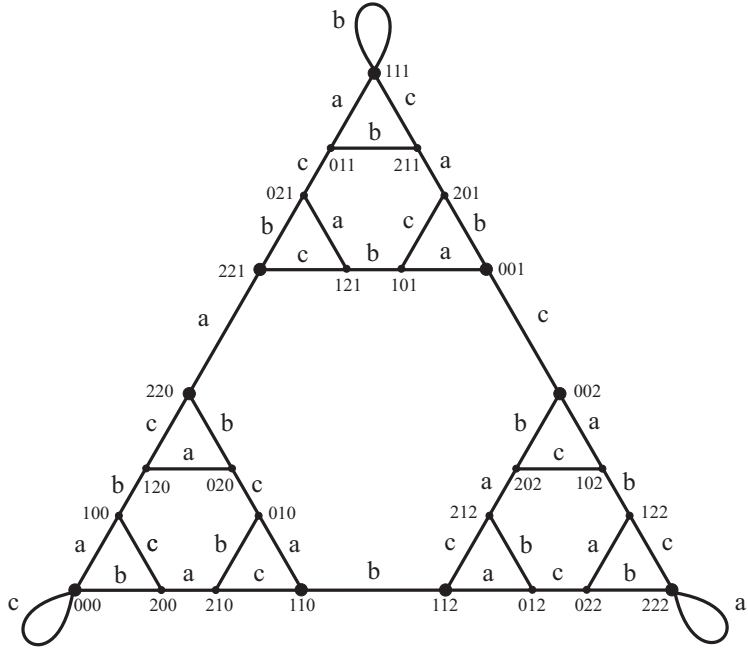}
\end{center}
 \caption{$\Gamma_3$, the Shcreier graph of $H^{(3)}$ at level 3}
 \label{f:g3}
\end{figure}

\begin{theorem}\label{t:main}
For $n \geq 1$, the spectrum of the graph $\Gamma_n$, as a set, has
$3 \cdot 2^{n-1}-1$ elements and is equal to
\[
 \{3\} \ \cup \ \bigcup_{i=0}^{n-1} f^{-i}(0) \ \cup \
                \bigcup_{j=0}^{n-2} f^{-j}(-2),
\]
where
\[ f(x) = x^2 - x - 3. \]

The multiplicity of the $2^i$ eigenvalues in $f^{-i}(0)$,
$i=0,\dots,n-1$ is $a_{n-i}$ and the multiplicity of the $2^j$
eigenvalues in $f^{-j}(-2)$, $j=0,\dots,n-2$ is $b_{n-j}$, where,
for $m \geq 1$,
\[ a_m = \frac{3^{m-1}+3}{2}, \qquad\qquad b_m = \frac{3^{m-1}-1}{2}. \]

The spectrum of $\Gamma$ (the Shcreier spectrum of $H^{(3)}$), as a
set, is equal to
\[
 \overline{\bigcup_{i=0}^\infty f^{-i}(0)} .
\]

It consists of a set of isolated points $I= \bigcup_{i=0}^\infty
f^{-i}(0)$ and its set of accumulation points $J$. The set $J$ is a
Cantor set and is the Julia set of the polynomial $f$.

The KNS spectral measure is discrete and concentrated on the the set
\[ \bigcup_{i=0}^\infty f^{-i}\{0,-2\}. \]
The KNS measure of each eigenvalue in $f^{-i}\{0,-2\}$ is
$\frac{1}{6 \cdot 3^i}$, $i=0,1,\dots$~.
\end{theorem}

Note that here and thereafter, for a map $\varphi$ and positive
integer $i$, we denote by $\varphi^{i}$ the $i$-th iterate $\varphi
\circ \cdots \circ \varphi$ of $\varphi$ under composition. For a
set $A$, $\varphi^{-i}(A)$ denotes the set of inverse images of $A$
under $\varphi^i$.

Before we move into more details, let us mention that, given the
relative simplicity of the graphs $\Gamma_n$, one may likely obtain
the same result by using some other techniques (in particular those
developed by Teplyaev and
Malozemov~\cite{teplyaev:gasket,malozemov-t:self-similarity}), but
the use of self-similar groups brings new ideas to spectral theory
that, we believe, could be successfully utilized in many other
situations.

As an added benefit, the Hanoi Tower groups $H^{(k)}$, $k \geq 3$,
are extremely interesting algebraic objects which have some unusual
properties (from the point of view of classical group theory). For
instance, $H^{(3)}$ is a finitely generated branch
group~\cite{grigorchuk-s:standrews} that has the infinite dihedral
group as a quotient (it is the first known example of a branch group
generated by a finite automaton that has this property) and is
therefore not a just infinite group (an example of a branch group
generated by a finite automaton that has the infinite cyclic group
as a quotient was given in~\cite{grigorchuk-d:fa}). Further,
$H^{(3)}$ is amenable but not elementary, and not even
subexponentially, amenable group (no familiarity with the
amenability property is needed to follow the content of this
manuscript; interested reader is referred
to~\cite{greenleaf:b-means} and~\cite{ceccherini-g-h:paradoxical}
for more background on amenability). The closure of $H^{(3)}$ in the
group of ternary tree automorphisms $\operatorname{Aut}(X^*)$ is a
finitely constrained group (a group of tree automorphisms defined by
a finite set of forbidden patterns). Finally, we mention that
$H^{(3)}$ is (isomorphic to) the iterated monodromy group of the
post-critically finite rational map $z \mapsto z^2 - \frac{16}{27z}$
(this map has already been studied from the point of view of complex
dynamics~\cite{devaney-l:16/27}). Some of the results listed above
were announced in~\cite{grigorchuk-n-s:oberwolfach1}
and~\cite{grigorchuk-n-s:oberwolfach2} and, along with other new
results, will be the subject of a more thorough treatment in a
subsequent paper~\cite{grigorchuk-n-s:hanoi}, written jointly with
Nekrashevych.

\section{Group theoretic model of Hanoi Tower Game and associated Schreier graphs}

We first describe briefly the Hanoi Towers Game (see~\cite{hinz:ens}
for more thorough description and historical details) and then
provide a model employing a group action on a rooted $k$-ary tree.

Fix an integer $k$, $k \geq 3$. The Hanoi Towers Game is played on
$k$ pegs, labeled by $0,1,\dots,k-1$, with $n$ disks, labeled by
$1,2,\dots,n$. All $n$ disks have different size and the disk labels
reflect the relative size of each disk (the disk labeled by 1 is the
smallest disk, the disk labeled by 2 is the next smallest, etc). A
configuration is (by definition) any placement of the $n$ disks on
the $k$ pegs in such a way that no disk is below a larger disk
(i.e., the size of the disks placed on any single peg is decreasing
from the bottom to the top of the peg). In a single step one may
move the top disk from one peg to another peg as long as the newly
obtained placement of disks is a configuration. Note that this
implies that, given two pegs $x$ and $y$, there is only one possible
move that involves these two pegs (namely the smaller of the two
disks on top of the pegs $x$ and $y$ may be moved to the other peg).
Initially all disks are on peg 0 and the object of the game is to
move all of them to peg 1 in the smallest possible number of steps.

Fix a $k$-letter alphabet $X=\{0,1,\dots,k-1\}$. The set $X^*$ of
(finite) words over $X$ has a rooted $k$-ary tree structure. The
empty word is the root, and the children of the vertex $u$ are the
$k$ vertices $ux$, for $x \in X$. The set $X^n$ of words of length
$n$ constitutes level $n$ in the tree $X^*$.

The $k$-ary rooted tree $X^*$ has self-similar structure. Namely,
for every vertex $u$, the tree $X^*$ is canonically isomorphic to
the subtree $uX^*$ hanging below vertex $u$ consisting of vertices
with prefix $u$. The canonical isomorphism $\phi_u:X^* \to uX^*$ is
given by $\phi_u(w) = uw$, for $w$ in $X^*$.

An automorphism $g$ of the $k$-ary tree $X^*$ induces a permutation
$\pi_g$ of the alphabet $X$, by setting $\pi_g(x)=g(x)$, for $x \in
X$. The permutation $\pi_g \in \operatorname{Sym}(X)$ is called the
root permutation of $g$. For a letter $x$ in $X^*$, the section of
$g$ at $x$ is the automorphism $g_x$ of $X^*$ defined by
\[ g_x(w) = \phi^{-1}_{g(x)}g\phi_x(w), \]
for $w$ in $X^*$. In other words, the section $g_x$ of $g$ at $x$
acts on the word $w$ exactly as $g$ acts on the tail (behind $x$) of
the word $xw$ and we have
\[ g(xw) = \pi_g(x) g_x(w), \]
for $w$ in $X^*$. Thus every tree automorphism $g$ of $X^*$ can be
decomposed as
\[
g = \pi_g \ (g_0,g_1,\dots,g_{k-1}),
\]
where $g_x$, $x=0,\dots,k-1$ are the sections of $g$ and $\pi_g$ is
the root permutation of $X$. Algebraically, this amounts to
decomposition of the automorphism group $\operatorname{Aut}(X^*)$ of
the tree $X^*$ as the semidirect product $\operatorname{Sym}(X)
\ltimes (\operatorname{Aut}(X^*) \times \cdots \times
\operatorname{Aut}(X^*))$. The direct product
$(\operatorname{Aut}(X^*) \times \cdots \times
\operatorname{Aut}(X^*))$ ($k$ copies) is the (pointwise) stabilizer
of the first level of $X^*$. Each factor in the first level
stabilizer $\operatorname{Aut}(X^*) \times \cdots \times
\operatorname{Aut}(X^*)$ acts on the corresponding subtree of $X^*$
hanging below the root (the factor corresponding to $x$ acts on
$xX^*$) and the symmetric group $\operatorname{Sym}(X)$ acts by
permuting these $k$ trees.

For any permutation $\pi$ in $\operatorname{Sym}(X)$ define a
$k$-ary tree automorphism $a=a_\pi$ by
\[
 a = \pi (a_0,a_1,\dots,a_{k-1}),
\]
where $a_i=a$ if $i$ is fixed by $\pi$ and and $a_i=1$ otherwise.

For instance, for the transposition $(ij)$, the action of $a_{(ij)}$
on $X^*$ is given recursively by $a_{ij}(\emptyset)=\emptyset$ and,
for $x \in X$ and $w \in X^*$,
\[
 a_{ij}(xw) =
  \begin{cases}
   jw, & x=i \\
   iw, & x=j \\
   xa_{ij}(w), & x \neq i, \ x \neq j.
  \end{cases}
\]
Thus, $a_{(ij)}$ ``looks'' for the first (leftmost) occurrence of
one of the letters in $\{i,j\}$ and replaces it with the other
letter. If none of the letters $i$ or $j$ appears, $a_{(ij)}$ leaves
the word unchanged.

The configurations in Hanoi Towers Game on $k$ pegs and $n$ disks
are in bijective correspondence with the words of length $n$ over
$X$ (vertices at level $n$ in the $k$-ary tree). The word
$x_1x_2\dots x_n$ represents the unique configuration in which disk
$i$ is on peg $x_i$ (once the location of each disk is known, there
is only one way to order them on their respective pegs).

The automorphism $a_{(ij)}$ of $X^*$ acts on the set of
configurations by applying a move between peg $i$ and peg $j$. To
apply a move between these two pegs one has to move the smallest
disk that is on top of either of these two pegs to the other peg.
Let $w=x_1x_2 \dots x_n$ be a word over $X$ and let the leftmost
occurrence of one of the letters $i$ or $j$ in $w$ happen at
position $m$. This means that none of the disks $1,2,\dots,m-1$ is
on peg $i$ or peg $j$ and disk $m$ is the smallest disk that appears
on either peg $i$ or $j$. Further, disk $m$ appears on peg $x_m
\in\{i,j\}$ and needs to be moved to peg $\overline{x}_m$, where
$\overline{x}_m=j$ if $x_m=i$ and $\overline{x}_m=i$ if $x_m=j$. The
word $a_{(ij)}(w)$ is obtained from $w$ by changing the leftmost
occurrence of one of the letters $i$ or $j$ to the other letter,
i.e., by changing the letter $x_m$ to $\overline{x}_m$, which
exactly corresponds to the movement of disk $m$ from peg $x_m$ to
peg $\overline{x}_m$. Note that if both peg $i$ and peg $j$ are
empty there is no occurrence of either $i$ or $j$ in $w$ and in such
a case $a_{(ij)}(w)=w$, i.e., the configuration is not changed after
a move between peg $i$ and peg $j$ is applied.

Hanoi Towers group on $k$ pegs is the group
\[
 H^{(k)} = \langle \{ a_{(ij)} \mid 0\leq i < j \leq k-1 \} \rangle
\]
of $k$-ary tree automorphisms generated by the automorphisms
$a_{(ij)}$, $0\leq i < j \leq k-1$, corresponding to the
transpositions in $\operatorname{Sym}(X)$. We often drop the
superscript in $H^{(k)}$ in case $k=3$, i.e., we set $H = H^{(3)}$.
Therefore, if we denote $a_{(01)} = a$, $a_{(02)} = b$, and
$a_{(12)} = c$, the Hanoi Towers group on 3 pegs is the group
\[ H = \langle a, b, c \rangle \]
generated by the ternary tree automorphisms $a$, $b$ and $c$,
defined by
\begin{align*}
 a &= (01) \ (1, 1, a) \\
 b &= (02) \ (1, b, 1) \\
 c &= (12) \ (c, 1, 1).
\end{align*}

The Schreier graph of the action of $H^{(k)}$ on $X^*$ on level $n$
is the regular graph $\Gamma_n^{(k)}$ of degree $k(k-1)/2$ whose
vertices are the words of length $n$ over $X$, and in which every
pair of vertices $u$ and $v$ for which $a_{(ij)}(u)=v$ (or,
equivalently, $a_{(ij)}(v)=u$) is connected by an edge labeled by
$a_{(ij)}$. In other words, two configurations in the Hanoi Towers
Game on $k$ pegs and $n$ disks are connected by an edge labeled by
$a_{(ij)}$ exactly when they can be obtained from each other by
applying a move between peg $i$ and peg $j$. We often drop the
superscript in $\Gamma_n^{(k)}$ in case $k=3$, i.e., we set
$\Gamma_n = \Gamma_n^{(3)}$. For instance, the graph $\Gamma_3$ is
given in Figure~\ref{f:g3}. Graphs closely related to
$\Gamma_n^{(k)}$, usually called Hanoi Towers graphs, are standard
feature in many works on Hanoi Tower Game, with a small (but
ultimately important) difference. Namely, in our setting, if a word
$w$ does not contain any letter $i$ or letter $j$, there is a loop
at $w$ in $\Gamma_n^{(k)}$ labeled by $a_{(ij)}$. The graph
corresponding to Hanoi Towers Game on 3 pegs and 3 disks that
appears in~\cite{hinz:ens} (and many other references) does not have
loops at the three ``corners''. Note that our loops do not change
the distances (or the diameter) in the graphs, but provide
additional level of regularity that is essentially used in our
considerations.

The boundary $\partial X^*$ of the tree $X^*$ is the set of all
right infinite words over $X$. Since automorphisms of the tree $X^*$
preserve the prefixes of words (if two words share a common prefix
of length $m$, so do their images), the action of any group of tree
automorphisms can be naturally extended to an action on the boundary
of the tree. The orbit of the word $0^\infty =000\dots$ in $\partial
X^*$ under the action of $H^{(k)}$ consists of all right infinite
words that end in $0^\infty$. The orbital Schreier graph
$\Gamma^{(k)}=\Gamma_{0^\infty}^{(k)}$ is the countable graph whose
vertices are the words in the orbit $H^{(k)}(0^\infty)$, and in
which a pair of infinite words $u0^\infty$ and $v0^\infty$, with
$|u|=|v|$, is connected by an edge labeled by $a_{(ij)}$ exactly
when $a_{(ij)}(u)=v$ (or, equivalently, $a_{(ij)}(v)=u$).  Note
that, in terms of the Hanoi Towers Game on $k$ pegs, the word
$0^\infty$ corresponds to the infinite configuration in which
countable many disks labeled by $1,2,3,\dots$ are placed on peg 0,
the vertices in $\Gamma^{(k)}$ are the infinite configurations that
can be reached from $0^\infty$ by using finitely many legal moves
and two infinite configurations are connected by an edge labeled by
$a_{(ij)}$ if one can be obtained from the other by applying a move
between peg $i$ and peg $j$. When $k=3$ we often omit the
superscript in $\Gamma^{(3)}$ and set $\Gamma=\Gamma^{(3)}$.

\section{Schreier spectrum of $H$}

The action of $H$ on level $n$ induces a permutational
$3^n$-dimensional representation $\rho_n: H \to GL(3^n,{\mathbb C})$
of $H$. Denote $\rho_n(a)=a_n$, $\rho_n(b)=b_n$ and $\rho_n(c)=c_n$.
The representation $\rho_n$ can be recursively defined by
\begin{gather}
 a_0=b_0=c_0=[1] \label{e:rec-generators} \\
 a_{n+1} =
 \begin{bmatrix}
  0_n & 1_n & 0_n \\
  1_n & 0_n & 0_n \\
  0_n & 0_n & a_n
 \end{bmatrix} \qquad
 b_{n+1} =
 \begin{bmatrix}
  0_n & 0_n & 1_n \\
  0_n & b_n & 0_n \\
  1_n & 0_n & 0_n
 \end{bmatrix} \qquad
 c_{n+1} =
 \begin{bmatrix}
  c_n & 0_n & 0_n \\
  0_n & 0_n & 1_n \\
  0_n & 1_n & 0_n
 \end{bmatrix}, \notag
 \end{gather}
where $0_n$ and $1_n$ are the zero and the identity matrix,
respectively, of size $3^n \times 3^n$.

The matrix $\Delta_n=a_n+b_n+c_n$ is the adjacency matrix of the
Schreier graph $\Gamma_n$ and is defined by
\[
 \Delta_0=[3], \qquad\qquad
 \Delta_{n+1} =
 \begin{bmatrix}
  c_n & 1_n & 1_n \\
  1_n & b_n & 1_n \\
  1_n & 1_n & a_n
 \end{bmatrix}.
\]
The spectrum of $\Gamma_n$ is the set of $x$ values for which the
matrix
\[ \Delta_n(x) = a_n+b_n+c_n-x \]
is not invertible. We introduce another real parameter (besides $x$)
and additional operator $d_n$ defined by
\[ d_{n+1} = \begin{bmatrix}
  0_n & 1_n & 1_n \\
  1_n & 0_n & 1_n \\
  1_n & 1_n & 0_n
 \end{bmatrix},
\]
and, for $n \geq 1$, consider the 2-parameter pencil $\Delta_n(x,y)$
of $3^n \times 3^n$ matrices given by
\[ \Delta_n(x,y) = a_n+b_n+c_n-x +(y-1)d_n, \]
i.e.,
\begin{equation}\label{e:delta-xy}
 \Delta_n(x,y) = \begin{bmatrix}
  c-x & y & y \\
  y & b-x & y \\
  y & y & a-x
 \end{bmatrix}
\end{equation}
(observe that here and thereafter we drop the index from $a_n$,
$b_n$, etc., in order to keep the notation less cumbersome). Instead
of trying to determine the values of $x$ for which
$\Delta_n(x)=\Delta_n(x,1)$ is not invertible we will find all pairs
$(x,y)$ for which $\Delta_n(x,y)$ is not invertible (call this set
of points in the plane the auxiliary spectrum). This seemingly
unmotivated excursion to a higher dimension actually comes naturally
(see the comments after Proposition~\ref{p:recursion}).

For $n \geq 1$, let
\[ D_n(x,y) = \det(\Delta_n(x,y)).\]
We provide a recursive formula for the determinant $D_n(x,y)$.

\begin{proposition}\label{p:recursion}
We have
\[
 D_1(x,y) = -(x-1-2y)(x-1+y)^2
\]
and, for $n \geq 2$,
\begin{equation}\label{e:recursion}
 D_n(x,y) = (x^2 - (1+y)^2)^{3^{n-2}} (x^2-1+y-y^2)^{2 \cdot 3^{n-2}} D_{n-1}(F(x,y)),
\end{equation}
where $F: {\mathbb R}^2 \to {\mathbb R}^2$ is the 2-dimensional
rational map given by
\[ F(x,y) = (x',y'), \]
and the coordinates $x'$ and $y'$ are given by
\[
 x' =
 x+ \frac{2y^2(-x^2+x+y^2)}{(x-1-y)(x^2-1+y-y^2)}
\]
and
\[
 y' = \frac{y^2(x-1+y)}{(x-1-y)(x^2-1+y-y^2)}.
\]
\end{proposition}

\begin{proof}
By expanding the block matrix~\eqref{e:delta-xy} for $\Delta_n(x,y)$
one more level (using the recursions for the generators $a$, $b$ and
$c$ provided in~\eqref{e:rec-generators}) we obtain
\begin{equation}
 \Delta_n(x,y) = \left[ \begin{array}{ccc|ccc|ccc}
 c - x & 0 & 0 & y & 0 & 0 & y & 0 & 0 \\
 0 & -x & 1 & 0 & y & 0 & 0 & y & 0 \\
 0 & 1 & -x & 0 & 0 & y & 0 & 0 & y \\
 \hline
 y & 0 & 0 & -x & 0 & 1 & y & 0 & 0 \\
 0 & y & 0 & 0 & b-x & 0 & 0 & y & 0 \\
 0 & 0 & y & 1 & 0 & -x & 0 & 0 & y \\
 \hline
 y & 0 & 0 & y & 0 & 0 & -x & 1 & 0\\
 0 & y & 0 & 0 & y & 0 & 1 & -x & 0\\
 0 & 0 & y & 0 & 0 & y & 0 & 0 & a-x \\
\end{array}
 \right].
\end{equation}
The last matrix is conjugate, by a permutational matrix that places
the entries involving $c$, $b$ and $a$ in the last three positions
on the diagonal, to the matrix
\begin{equation}\label{e:delta-bar}
 \overline{\Delta}_n(x,y) =
 \left[
 \begin{array}{cccccc|ccc}
  -x &  0 & 0  & y  & 1  & 0 & y & 0 & 0 \\
  0  & -x & 1  & 0  & y  &  0 & 0 & y & 0 \\
  0  &  1 & -x & 0  & 0  & y & 0 & 0 & y \\
  y  &  0 & 0  & -x & 0  & 1 & y & 0 & 0 \\
  1  &  y & 0  & 0  & -x & 0 & 0 & y & 0 \\
  0  &  0 & y  & 1  & 0  & -x& 0 & 0 & y \\
  \hline
  y  &  0 & 0  & y  & 0 & 0 & c-x & 0 & 0 \\
  0  &  y & 0  & 0  & y & 0 & 0 & b-x & 0 \\
  0  &  0 & y  & 0  & 0 & y & 0 & 0 & a-x
 \end{array}
 \right].
\end{equation}
Thus, we have
\[
 \overline{\Delta}_n(x,y) =
  \begin{bmatrix}
   M_{11} & M_{12} \\ M_{21} & M_{22}
  \end{bmatrix},
\]
where the subdivision into blocks is indicated
in~\eqref{e:delta-bar}. We can easily calculate the determinant of
the matrix $M_{11}$ (since each block of size $3^{n-1} \times
3^{n-1}$ in $M_{11}$ is a scalar multiple of the identity matrix)
\[
 \det(M_{11}) = (x - 1 - y)^{3^{n-2}} (x + 1 + y)^{3^{n-2}} (x^2-1+y-y^2)^{2 \cdot
 3^{n-2}}
\]
and the Schur complement (\cite{zhang-schur} provides a historical
overview of Schur complement as well as wide ranging applications)
of $M_{11}$ in $\overline{\Delta}_n(x,y)$
\[ M_{22}-M_{21}M_{11}^{-1}M_{12} =
  \begin{bmatrix}
  c-x' & y' & y' \\
  y' & b-x' & y' \\
  y' & y' & a-x'
  \end{bmatrix}.
\]

Therefore
\begin{align*}
 D_n(x,y) &= \det(\Delta_n(x,y)) = \det(\overline{\Delta}_n(x,y))
 = \det(M_{11}) \det(M_{22} - M_{21}M_{11}^{-1}M_{12}) = \\
 &= (x^2 - (1+y)^2)^{3^{n-2}}
 (x^2-1+y-y^2)^{2 \cdot 3^{n-2}} D_{n-1}(x',y'). \qedhere
\end{align*}

\end{proof}

We can indicate now the reason behind the introduction of the new
parameter $y$. If the same calculation performed during the course
of the proof of Proposition~\ref{p:recursion} were applied directly
to the matrix $\Delta_n(x)$, the  corresponding Schur complement
would have been equal to
\[
\begin{bmatrix}
  c-x'' & y'' & y'' \\
  y'' & b-x'' & y'' \\
  y'' & y'' & a-x''
  \end{bmatrix},
\]
where
\[
 x'' = x+ \frac{2(-x^2+x+1)}{(x-2)(x^2-1)}
 \qquad\text{and}\qquad
 y'' = \frac{x}{(x-2)(x^2-1)}.
\]
In particular, the blocks off the main diagonal are not identity
anymore (as they were in $\Delta_n(x)$). The new parameter $y$ just
keeps track of this change.

The recursion for $D_n(x,y)$ already gives a good way to calculate
the characteristic polynomial of $\Delta_n$ for small values of $n$
(it is easier to iterate the recursion 9 times than to try to
directly calculate this polynomial for a $3^{10} \times 3^{10}$
matrix).

Further, for small values of $n$ we can easily plot the curves in
${\mathbb R}^2$ along which $D_n(x,y)=0$ and get an idea on the
structure of the auxiliary spectrum. For instance, for $n=1,2,3,4$
these curves are given in Figure~\ref{f:hyperbolae}. It is already
apparent from these graphs that with each new iteration several
hyperbolae are added to the auxiliary spectrum.
\begin{figure}[!ht]
\begin{tabular}{cc}
\includegraphics[width=150pt]{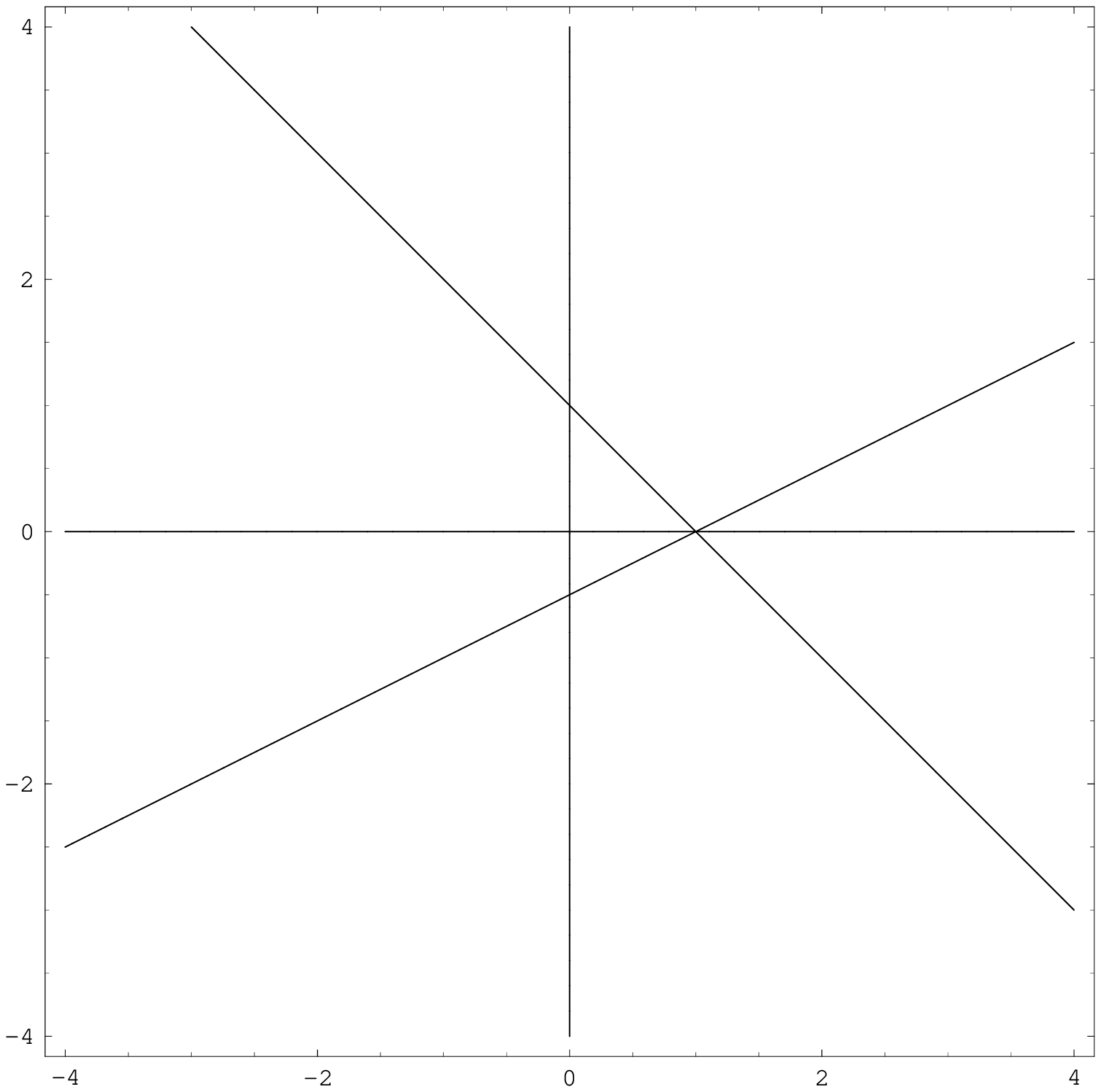} &
\includegraphics[width=150pt]{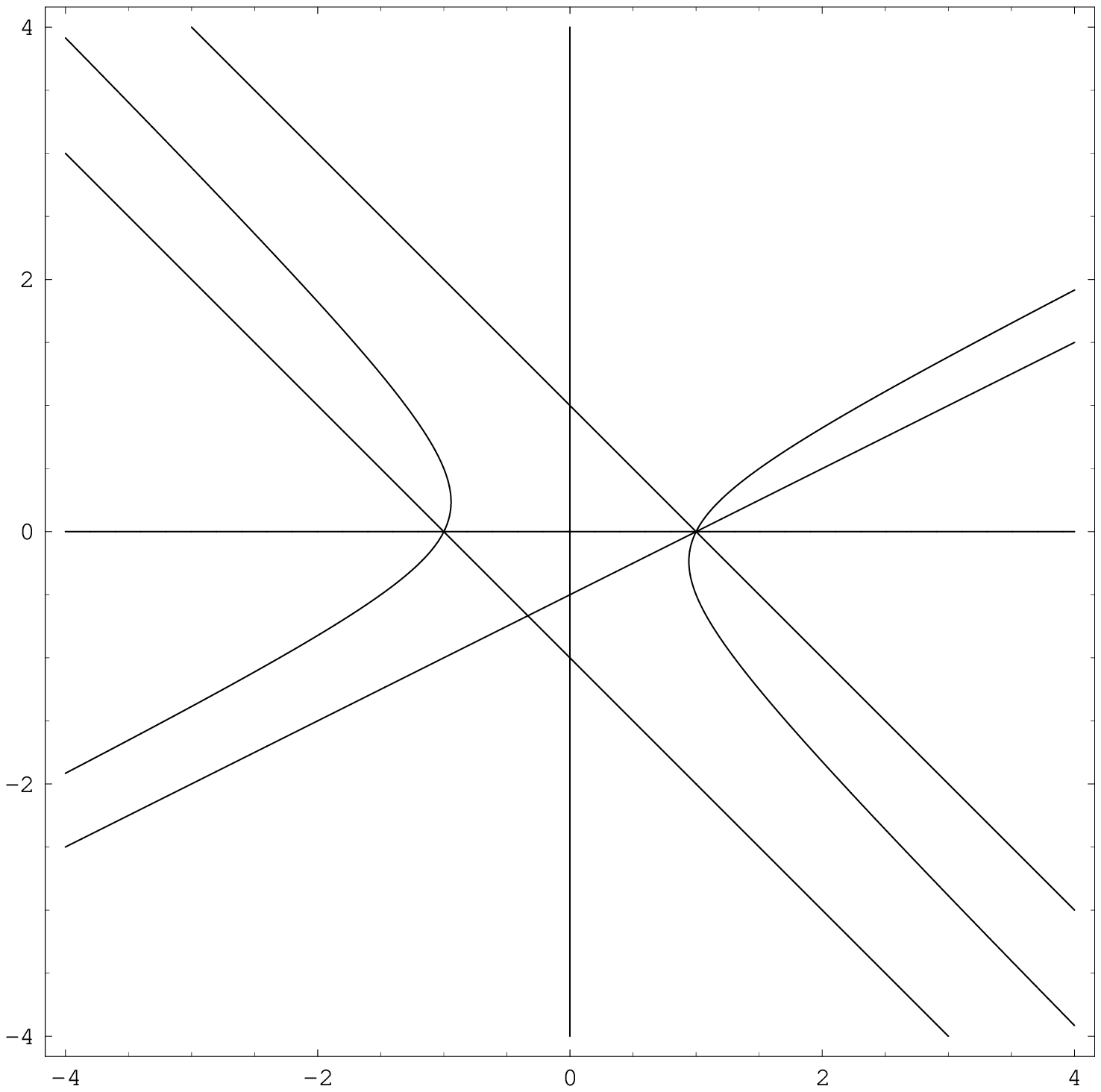} \\
\includegraphics[width=150pt]{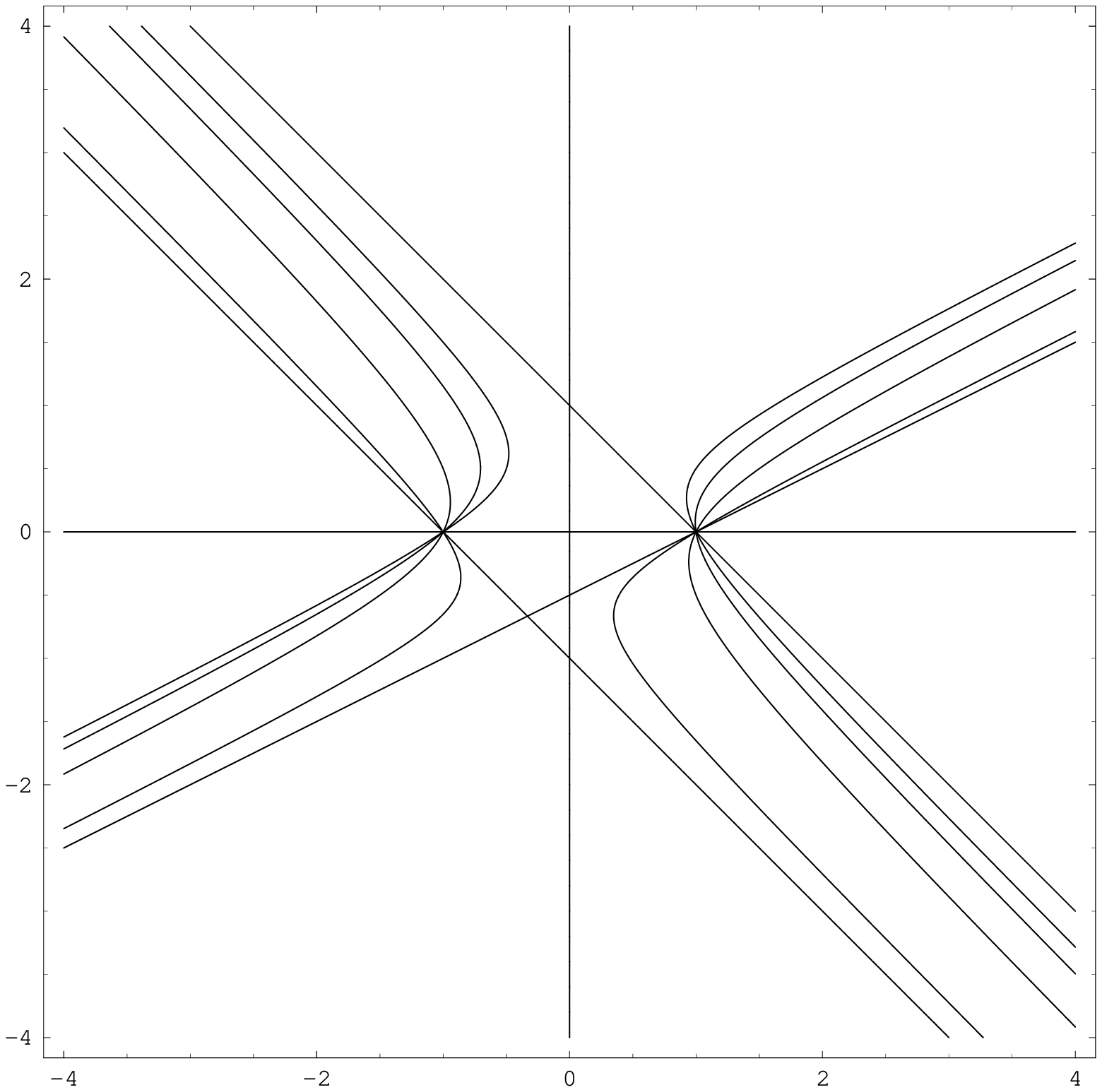} &
\includegraphics[width=150pt]{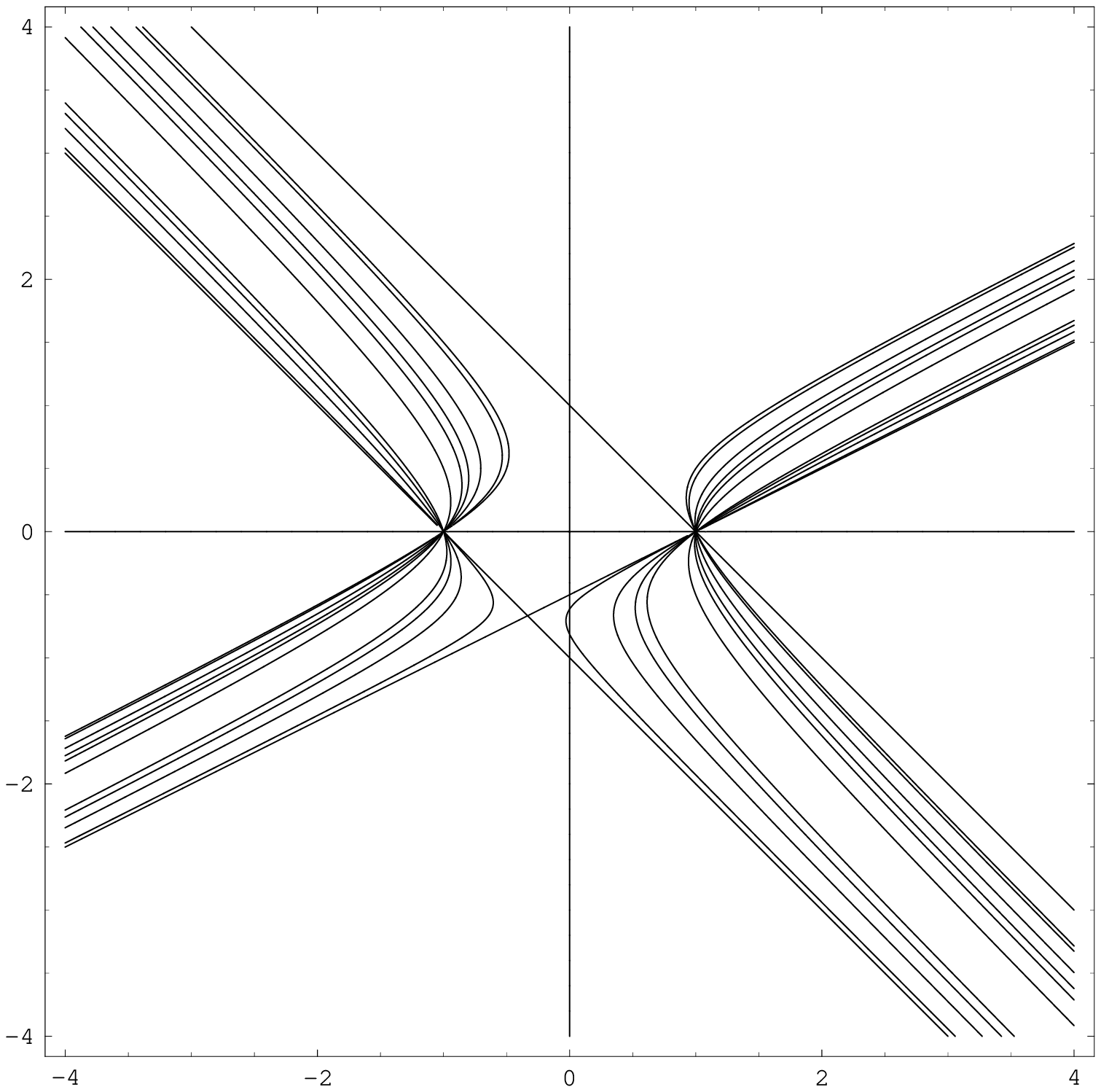}
\end{tabular}
\caption{Hyperbolae determining the auxiliary spectrum at level 1,
2, 3, and 4}\label{f:hyperbolae}
\end{figure}
The spectrum of $\Gamma_n$ is precisely the intersection of the line
$y=1$ and the auxiliary spectrum at level $n$. Note that other
examples in which the auxiliary spectrum is built from hyperbolae
already appear in~\cite{bartholdi-g:spectrum} (the first group of
intermediate growth, Gupta-Sidki group, Fabrykowski-Gupta group),
and there are examples where this is not the case (Basilica
group~\cite{grigorchuk-z:basilica2}, the iterated monodromy group of
the map $z \mapsto z^2+i$~\cite{grigorchuk-s-s:z2i}).

The next step is to provide precise description of the phenomenon we
just observed (the appearance of hyperbolae indicates that we may be
able to decompose $D_n(x,y)$ into factors of degree at most 2 over
${\mathbb R}$).

Define a transformation $\Psi: {\mathbb R}^2 \to {\mathbb R}$ by
\[
 \Psi(x,y) = \frac{x^2-1-xy-2y^2}{y}
\]
and a transformation $f: {\mathbb R} \to {\mathbb R}$ by
\[
 f(x) = x^2 - x - 3.
\]
First, we need some basic understanding of the dynamics of the
quadratic map $f(x)$. The critical point (the point where $f'(x)=0$)
of the map $f(x)$ is $1/2$. Therefore the critical value
$f(1/2)=-13/4$ is the unique value of $\theta$ for which
$f(x)=\theta$ has a double root. It is easy to check that
$f^{-1}[-2,3] = [-2,(1-\sqrt{5})/2] \cup [(1+\sqrt{5})/2,3]
\subseteq [-2,3]$. Since the critical value $-13/4 \not\in
f^{-1}[-2,3] \subseteq [-2,3]$ it follows that, for any value of
$\theta$ in $[-2,3]$, the entire backward orbit $f^{-i}(\theta)$ is
contained in $[-2,3]$ and the sets $f^{-i}(\theta)$,
$i=0,1,2,\dots$, consist of $2^i$ real numbers. Further, for such
$\theta$, the sets $f^{-i}(\theta)$ are mutually disjoint for
$i=0,1,\dots$, provided $\theta$ is not a periodic point (a point
$\zeta$ is periodic if $f^m(\zeta)=\zeta$ for some positive integer
$m$).

We come to a simple but crucial observation that lies behind all the
calculations that follow. It will eventually allow us to ``reduce
the dimension'' back to 1 and consider iterations of the
1-dimensional polynomial map $f(x)$ rather than the 2-dimensional
rational map $F(x,y)$.

\begin{lemma}
The 2-dimensional rational map $F$ is semi-conjugate to the
1-dimensional polynomial map $f$ through the map $\Psi$, i.e.,
\[ \xymatrix@C=8pt@R=8pt{
 & & {\mathbb R}^2 \ar[rr]^F \ar[dd]_\Psi && {\mathbb R}^2 \ar[dd]^\Psi
 \\
 \Psi(F(x,y)) = f(\Psi(x,y))
 \\
  & & {\mathbb R} \ar[rr]_f && {\mathbb R}
 }
\]
\end{lemma}

\begin{proof}
The claim can be easily verified directly.
\end{proof}

Let
\begin{align*}
 \Psi_\theta(x,y) &= x^2 - 1 -xy -2y^2 -\theta y = y(\Psi(x,y) - \theta), \\
 L(x,y) &= x-1-y, \\
 K(x,y) &= x^2-1+y-y^2, \\
 A_1(x,y) &= x-1+y.
\end{align*}
In order to simplify notation we sometimes write $P$ and $P'$
instead of $P(x,y)$ and $P(F(x,y))$.

\begin{lemma}\label{l:psi'}
Let $\theta \in [-2,3]$ and let $\theta_0$ and $\theta_1$ be the two
distinct real roots of $f(x)=\theta$. Then
\[ \frac{A_1}{LK}\Psi_{\theta_0} \Psi_{\theta_1} = \Psi_\theta'. \]
\end{lemma}
\begin{proof}
We have
\begin{align*}
 \Psi_\theta' &= y'(\Psi' - \theta) = \frac{y^2 A_1}{LK}(f(\Psi) - \theta)
 = \frac{y^2 A_1}{LK}(\Psi - \theta_0)(\Psi - \theta_1) = \frac{A_1}{LK}\Psi_{\theta_0} \Psi_{\theta_1}.
 \qedhere
\end{align*}
\end{proof}

Define the polynomial
\[
 D_0(x,y) = -(x-1-2y),
\]
two families of polynomials
\[
 A_n(x,y) =
 \begin{cases}
 x-1+y, & n=1 \\
 \prod_{\theta \in f^{-(n-2)}(0)} \Psi_\theta, & n \geq 2
 \end{cases}
\]
and
\[
 B_n(x,y) =
 \begin{cases}
 x+1+y, & n=2 \\
 \prod_{\theta \in f^{-(n-3)}(-2)} \Psi_\theta, & n \geq 3
 \end{cases},
\]
and two integer sequences
\[ a_n = \frac{3^{n-1}+3}{2}, \qquad\qquad b_n = \frac{3^{n-1}-1}{2}, \]
for $n \geq 1$.

Note that each factor $\Psi_\theta$ that appears above in $A_n$ and
$B_n$ is quadratic polynomial in ${\mathbb R}[x,y]$. It will be
shown that the hyperbolae $\Psi_\theta=0$ are precisely the
hyperbolae in the auxiliary spectrum.

\begin{lemma}
\begin{align*}
 D_0' &= \frac{D_0}{L}A_1, \\
 A_1' &= \frac{A_1}{K}A_2, \\
 A_n' &= \left(\frac{A_1}{LK}\right)^{2^{n-2}} A_{n+1}, \text{ for } n \geq 2, \\
 B_2' &= \frac{B_2}{K}B_3, \\
 B_n' &= \left(\frac{A_1}{LK}\right)^{2^{n-3}} B_{n+1}, \text{ for } n \geq 3.
\end{align*}
\end{lemma}
\begin{proof}
From Lemma~\ref{l:psi'} we obtain, for $n \geq 2$,
\[ A_n' = \prod_{\theta \in f^{-(n-2)}(0)} \Psi_\theta' =
\left(\frac{A_1}{LK}\right)^{2^{n-2}}\prod_{\theta \in
f^{-(n-1)}(0)} \Psi_\theta = \left(\frac{A_1}{LK}\right)^{2^{n-2}}
A_{n+1}. \] The claim involving $B_n'$ can be verified in a similar
manner. All the other claims do not involve $n$ and can be easily
verified directly.
\end{proof}

At this moment we can provide a factorization of the determinant
$D_n(x,y)$.

\begin{proposition}\label{p:factorization}
\begin{align}
 D_1(x,y) &= D_0 A_1^{a_1}, \notag \\
 D_n(x,y) &= D_0 A_1^{a_n}A_2^{a_{n-1}} \cdots A_n^{a_1} B_2^{b_n} B_3^{b_{n-1}} \cdots B_n^{b_2}, \text{ for }n \geq 2
 . \label{e:factorization}
\end{align}
\end{proposition}

\begin{proof}
The claim is correct for $n=1$. Assume the claim is correct for
$n-1$ (where $n \geq 2$). Then, by Proposition~\ref{p:recursion},
\begin{small}
\begin{align*}
 D_n(x,y) &= (B_2 L)^{3^{n-2}} K^{2 \cdot 3^{n-2}} D_{n-1}(F(x,y)) = \\
 &= (B_2 L)^{3^{n-2}} K^{2 \cdot 3^{n-2}} D_0' (A_1')^{a_{n-1}} \cdots (A_{n-1}')^{a_1} (B_2')^{b_{n-1}} \cdots (B_{n-1}'))^{b_2} = \\
 &= (B_2 L)^{3^{n-2}} K^{2 \cdot 3^{n-2}} \frac{D_0A_1}{L} \left(\frac{A_1}{LK}\right)^{m_n} \frac{A_1^{a_{n-1}}}{K^{a_{n-1}}}A_2^{a_{n-1}} \cdots A_n^{a_1} \frac{B_2^{b_{n-1}}}{K^{b_{n-1}}}B_3^{b_{n-1}} \cdots B_n^{b_2},
\end{align*}
\end{small}
where $m_n = (a_{n-2}+2a_{n-3} + \cdots+2^{n-3}a_1) +
(b_{n-2}+2b_{n-3}+\cdots+2^{n-4}b_2)$. It is easy to verify that
\begin{alignat*}{2}
 3^{n-2} &= m_n + 1,  &&a_n = m_n + a_{n-1} + 1 ,\\
 2\cdot 3^{n-2} &= m_n + a_{n-1} + b_{n-1} , \qquad\qquad &&b_n = b_{n-1}+3^{n-2} ,
\end{alignat*}
and therefore $D_n(x,y) = D_0A_1^{a_n}A_2^{a_{n-1}} \cdots A_n^{a_1}
B_2^{b_n}B_3^{b_{n-1}} \cdots B_n^{b_2}$.
\end{proof}

\begin{proof}[Proof of Theorem~\ref{t:main}]
A factorization of $D_n(x,y)$ is already provided in
Proposition~\ref{p:factorization}. The eigenvalues of $\Gamma_n$
correspond to the zeros of the polynomial $D_n(x,1)$. The claims in
Theorem~\ref{t:main} concerning the eigenvalues of $\Gamma_n$ follow
immediately once it is observed that
\begin{align*}
 \Psi(x,1) &= f(x), \\
 \Psi_\theta(x,1) &= f(x) -\theta, \\
 D_0(x,1) &= -(x-3), \\
 A_1(x,1) &= x, \\
 A_n (x,1)&= \prod_{\theta \in f^{-(n-2)}(0)} (f(x) - \theta), \ n \geq 2\\
 B_2(x,1) &= x+2, \\
 B_n(x,1) &= \prod_{\theta \in f^{-(n-3)}(-2)} (f(x) - \theta), \ n \geq
 3.
\end{align*}
Further, the forward orbit of 0 under $f$ escapes to $\infty$. Thus
$0$ is not periodic point and this implies that the sets $f^{-i}(0)$
are mutually disjoint for $i=0,1,2,\dots$. Similarly, since
$f(-2)=3$ and 3 is fixed point of $f$, the point $-2$ is not
periodic and the sets $f^{-i}(-2)$ are mutually disjoint for
$i=0,1,2,\dots$~. Thus the number of distinct eigenvalues of
$\Gamma_n$, for $n \geq 1$, is $1+(2^{n}-1)+(2^{n-1}-1) =3\cdot
2^{n-1} -1$.

Since $H$ is amenable (or more obviously, since the graph $\Gamma$
is amenable), the spectrum of $\Gamma$ is given by
(see~\cite{bartholdi-g:spectrum} for details)
\[
 \overline{\{3\} \cup \bigcup_{i=0}^\infty f^{-i}(0) \cup \bigcup_{i=0}^\infty f^{-i}(-2)}  .
\]
Recall that a periodic point $\zeta$ of $f$ is repelling if $\lvert
f'(\zeta) \rvert >1$ (see~\cite{devaney:b-graduate-chaotic}
or~\cite{beardon:b-iteration}for basic notions and results on
iteration of rational functions). Since $f(3)=3$ and $f'(3)=5$, the
point $3$ is a repelling fixed point for the polynomial $f$. This
implies that the backward orbit of $3$, which is equal to $\{3\}
\cup \bigcup_{i=0}^\infty f^{-i}(-2)$, is in the Julia set $J$ of
$f$ (see~\cite[Theorem~6.4.1]{beardon:b-iteration}). On the other
hand $0$ is not in the Julia set (its forward orbit escapes to
$\infty$) and therefore the set $I=\bigcup_{i=0}^\infty f^{-i}(0)$
is a countable set of isolated points that accumulates to the Julia
set $J$ of $f$. The spectrum of $\Gamma$, as a set, is therefore
equal to
\[
 \overline{\bigcup_{i=0}^\infty f^{-i}(0)}.
\]
The Julia set $J$ has the structure of a Cantor set since $f$ is
conjugate to the map $x \mapsto x^2-15/4$ and $-15/4<-2$
(see~\cite[Section~3.2]{devaney:b-graduate-chaotic}).

Recall that (see \cite{bartholdi-g:spectrum,grigorchuk-z:ihara}) the
KNS spectral measure $\nu$ of is limit of the counting measures
$\nu_n$ defined for $\Gamma_n$ ($\nu_n(B)=m_n(B)/3^n$, where
$m_n(B)$ counts, including the multiplicities, the eigenvalues of
$\Gamma_n$ in $B$). For the eigenvalues in $f^{-i}(0)$ we have
\[
 \lim_{n \to \infty} \frac{a_{n-i}}{3^n} = \lim_{n \to \infty} \frac{3^{n-i-1}+3}{2 \cdot 3^n} = \frac{1}{6 \cdot 3^i}.
\]
Since $b_n = a_n-2$, for $n \geq 1$, the density of the eigenvalues
in $f^{-i}(-2)$ is also $1/(6 \cdot 3^i)$. Since all these densities
add up to 1, the KNS spectral measure is discrete and concentrated
at these eigenvalues.
\end{proof}

\begin{corollary}
The characteristic polynomial $P_n(x)$ of the matrix $\Delta_n$
decomposes into irreducible factors over ${\mathbb Q}[x]$ as
follows:
\begin{alignat*}{7}
 P_0(x) &= -(x-3) \\
 P_1(x) &= -(x-3) x^2  \\
 P_2(x) &= -(x-3) x^3      &&(f(x))^2  &&             && &&(x+2) \\
 P_3(x) &= -(x-3) x^6      &&(f(x))^3  &&(f^2(x))^2   && &&(x+2)^4 \ &&g(x+2) \\
 P_4(x) &= -(x-3) x^{15}   &&(f(x))^6  &&(f^2(x))^3  &&(f^3(x))^2 &&
               (x+2)^{13}  &&(g(x+2))^4  &&g^2(x+2) \\
 \dots
\end{alignat*}
and in general
\begin{small}
\[
 P_n(x) = -(x-3)\ x^{a_n} (f(x))^{a_{n-1}} \cdots (f^{n-1}(x))^{a_1} \
 (x+2)^{b_n} (g(x+2))^{b_{n-1}} \cdots g^{n-2}(x+2)^{b_2},
\]
\end{small}
where $g(x) = x^2 - 5x + 5$.
\end{corollary}

\begin{proof}
Observe that $g(x+2) = f(x)+2$ (i.e., $f$ and $g$ are conjugate by
translation by 2). This implies that, for $n \geq 1$, $g^n(x+2) =
f^n(x)+2$. The correctness of the factorization directly follows
from Proposition~\ref{p:factorization} (for $y=1$).

Since $f(x)$ and $g(x)$ are irreducible quadratic polynomials and
their discriminants are odd, it follows that, for all $n \geq 1$,
$f^n(x)$ and $g^n(x)$ (and therefore also $g^n(x+2)$) are
irreducible (see~\cite[Theorem 2]{ayad-m:quadratic}).
\end{proof}

Before we conclude the section, it is perhaps necessary to make the
following technical remark. During the course of the proof of
Proposition~\ref{p:recursion} we used the formula $D_n =
\det(M_{11}) \det(M_{22}-M_{21}M_{11}^{-1}M_{12})$ to establish a
recursion for $D_n$. However, this formula may be used only when
$M_{11}$ is invertible (thus, we must stay off the lines $x-1-y=0$,
$x+1+y=0$ and the hyperbola $x^2-1+y-y^2$). The problem is
compounded by the fact that the formula for $D_n(x,y)$ needs to be
iterated, which then means that it may not be used at points on the
pre-images of the above curves under the 2-dimensional map $F$.
Nevertheless, the factorization formulas~\eqref{e:factorization}
provided in Proposition~\ref{p:factorization} are correct on some
open set in the plane (avoiding these pre-images) and since these
formulas represents equalities of polynomials in two variables they
are correct on the whole plane.

\section{Some additional comments}

We used the spectra of the level Schreier graphs $\Gamma_n$ to
calculate the spectrum of the boundary Schreier graph
$\Gamma_{0^\infty}$ and we called this spectrum the Schreier
spectrum of $H$. Since the boundary $\partial X^*$ is uncountable
and $H$ is countable, there are uncountably many orbits of the
action of $H$ on the boundary and there are clearly non-isomorphic
ones among them (for instance $\Gamma_{0^\infty}$ has exactly one
loop, while $\Gamma_{(012)^\infty}$ has none). However, the spectra
of all these boundary Schreier graphs coincide and are equal to the
closure of the union of the spectra of $\Gamma_n$. This follows
from~\cite[Proposition~3.4]{bartholdi-g:spectrum}, since $H$ acts
transitively on each level of the tree $X^*$ and is amenable (the
amenability of groups generated by bounded automata is proven
in~\cite{bartholdi-k-n-v:bounded}).

The boundary $\partial X^*$ supports a canonical invariant measure
$\mu$ defined as the Bernoulli product measure on $\partial X^*$
induced by the uniform measure on the finite set $X$. Thus we can
associate to any group $G$ of tree automorphisms a unitary
representation $\rho$ of $G$ on $L_2(\partial X^*,\mu)$, defined by
$(\pi(g)(\alpha))(w) = \alpha(g^{-1}(w))$, for $g \in G$, $\alpha
\in L_2(\partial X^*,\mu)$ and $w \in
\partial X^*$. If $S=S^{-1}$ is a finite symmetric set of generators of $G$,
we associate to the representation $\pi$ a Hecke type operator
defined by $T_S = \frac{1}{\lvert S \rvert} \sum_{s \in S} \pi(s)$.
It follows from~\cite[Theorem~3.6]{bartholdi-g:spectrum} and the
amenability of $H$ that, when $G=H$ and $S=\{a,b,c\}$, the spectrum
of $T_S$ is equal to the Schreier spectrum of $H$ re-scaled by
$1/3$.

Speaking of amenability, the question of amenability of $H^{(k)}$ is
open for $k \geq 4$. The fact that many questions, including
diameter, spectra, average distance (see~\cite{hinz-s:average}) can
be answered in case $k=3$, but are still open in case $k \geq 4$ may
very well be related to the possibility that the groups $H^{(k)}$
are not amenable for $k \geq 4$, even though the graphs
$\Gamma^{(k)}$ are amenable (an obvious sequence $\{F_n\}$ of
F{\o}lner sets  is obtained by declaring $F_n$ to consist of all
right infinite words in which all symbols after position $n$ are
equal to 0).

Thus, both the Hanoi Towers groups $H^{(k)}$, $k \geq 4$, and the
associated Schreier graphs are very interesting objects and, in our
opinion, more work on questions related to algebraic and
combinatorial properties of the groups and analysis and random walks
on these groups and graphs is certainly needed.

\section*{Acknowledgment}

We would like to thank Isaac Newton Institute, Cambridge UK, and
their staff for the hospitality and support during the Program
``Analysis on Graphs and its Applications'' (Spring 2007). We would
also like to thank the Organizers B.~M.~Brown, P.~Exner,
P.~Kuchment, and T.~Sunada for giving us the opportunity to
participate in the Program.


\def\cprime{$'$}
\providecommand{\bysame}{\leavevmode\hbox
to3em{\hrulefill}\thinspace}
\providecommand{\MR}{\relax\ifhmode\unskip\space\fi MR }
\providecommand{\MRhref}[2]{%
  \href{http://www.ams.org/mathscinet-getitem?mr=#1}{#2}
} \providecommand{\href}[2]{#2}

\end{document}